\documentclass[final,11pt]{elsarticle}

\usepackage{amssymb,amsmath,amssymb}

\usepackage{mathrsfs}
\usepackage{float}

\usepackage{amsthm}
\usepackage{graphics}
\usepackage{caption2}
\usepackage{psfrag}
\usepackage{showlabels}
\usepackage{color}
\usepackage{amsmath}
\usepackage{bm}
\usepackage{mathtools}
\usepackage{tensor}
\usepackage{subfigure}
\usepackage[normalem]{ulem}
\usepackage[linesnumbered,boxed,ruled,commentsnumbered]{algorithm2e}
\usepackage{graphicx}
\usepackage{epstopdf}
\usepackage{array}
\usepackage{makecell}
\usepackage[english]{babel}
\usepackage{booktabs}
\usepackage{threeparttable}

\usepackage{amsfonts}
\usepackage{amssymb}
\usepackage{graphicx}
\usepackage{multirow}
\usepackage{exscale}
\usepackage{relsize}
\usepackage[normalem]{ulem}
\usepackage{color}
\usepackage{fullpage}
\usepackage{lineno}
\allowdisplaybreaks[4]

\newcommand{\x}{\mathbf{x}}

\newcommand{\be}{\begin{equation}}
\newcommand{\ee}{\end{equation}}
\newcommand{\ba}{\begin{array}}
\newcommand{\ea}{\end{array}}
\newtheorem{theorem}{Theorem}

\newtheorem{remark}{Remark}
\allowdisplaybreaks[4]

\journal{Journal of Computational and Applied
	Mathematics}

\begin{document}
\begin{frontmatter}

\title{A high-order multi-time-step scheme for bond-based peridynamics}

\author{ Chenguang Liu$^a$, Jie Sun $^a$, Hao Tian$^{a*}$,  WaiSun Don $^a$, Lili Ju $^b$}

\address{
$^a$ School of Mathematical Science, Ocean University of China, Qingdao, Shandong 266100, China \\
$^b$Department of Mathematics, University of South Carolina, Columbia, South Carolina 29208, USA\\
 }

\begin{abstract}
A high-order multi-time-step (MTS) scheme for the bond-based peridynamic (PD) model, an extension of classical continuous mechanics widely used for analyzing discontinuous problems like cracks, is proposed. The MTS scheme discretizes the spatial domain with a meshfree method and advances in time with a high-order Runge-Kutta method. To effectively handle discontinuities (cracks) that appear in a local subdomain in the solution, the scheme employs the Taylor expansion and Lagrange interpolation polynomials with a finer time step size, that is, coarse and fine time step sizes for smooth and discontinuous subdomains, respectively, to achieve accurate and efficient simulations. By eliminating unnecessary fine-scale resolution imposed on the entire domain, the MTS scheme outperforms the standard PD scheme by significantly reducing computational costs, particularly for problems with discontinuous solutions, as demonstrated by comprehensive theoretical analysis and numerical experiments.
\end{abstract}


\end{frontmatter}

\section{Introduction}
Classical continuum mechanics relies on partial differential equations to study movement systems; however, it has limitations in handling problems with discontinuities. To overcome this deficiency, Silling \cite{SA} developed the peridynamic approach, which utilizes nonlocal integral equations. Initially, peridynamic models were bond-based \cite{AXS}, requiring materials to have fixed Poisson's ratios. Later, state-based models \cite{SA2} were introduced, eliminating the limitation of Poisson's ratio. The versatility of peridynamic models has been demonstrated in various practical applications, such as deformation analysis of composite materials \cite{GRUA, YNE, VB}, investigation of corrosion phenomena \cite{SZF, ZSJ, JSM, SZS, SZJF}, prediction of damage \cite{CSE, XXXE2, DW, MKA, PZ}, and simulation of crack propagation in diverse materials \cite{SMY, YXG, SE, SYL, PS}.

Various spatial discretization approaches have been employed in peridynamics, including meshfree, collocation, finite element, and finite difference methods \cite{XQ, TCP,JY, JMY, SY}. Researchers have also explored algorithms based on fast Fourier transform \cite{SF, SL} and coupling methods \cite{MT, TM, JFZJ, FGY, HHHY, MPRM} to alleviate the computational burden. For temporal discretization, adaptive dynamic relaxation (ADR) and velocity Verlet (VV) methods are commonly used methods for solving quasi-static and dynamic problems, respectively, providing accurate and efficient solutions in peridynamics.

However, when dealing with large and complex nonlocal problems, the computational burden of the aforementioned numerical methods can be significant, especially when using a small step to achieve temporal accuracy. As a solution, local-time-stepping schemes have gained popularity in recent years. One notable example is the local-time-stepping (LTS) scheme \cite{TD}, which builds upon the strong second-order stability preserving Runge-Kutta (SSP-RK) method for shallow-water equations. Similarly, the explicit local-time-stepping scheme \cite{KL} employs higher-order Runge-Kutta methods to provide greater temporal accuracy for solving multi-dimensional systems of conservation laws. Another approach is multi-time-step integration, based on the velocity Verlet method, improving computational efficiency by utilizing variable time step sizes for different components \cite{LP}. However, several outstanding issues still need to be addressed.
\begin{enumerate}
	\item Previous studies on local-time-stepping schemes have mainly focused on nonlocal models other than PD.
	\item Only the first- or second-order temporal discretization schemes were introduced for PD.
\end{enumerate}


This article presents a multi-time-step scheme (MTS) for peridynamics based on high-order Runge-Kutta methods. The proposed approach extends the local-time-stepping solution to the peridynamics model, adapting the fine time step size near the crack, reducing computational costs, and ensuring accuracy. By incorporating higher-order time-stepping schemes, the proposed method achieves comparable accuracy to lower-order time-stepping schemes, such as adaptive dynamic relaxation and velocity Verlet algorithms, in a shorter time, thus reducing the overall computational cost of peridynamics simulations. We validate the feasibility and performance of the proposed algorithm through extensive numerical examples, demonstrating its efficacy in solving peridynamics problems.


The article is organized as follows: Section 2 reviews the peridynamic model and the spatial discretization. Section 3 analyzes the MTS scheme based on the third- and fourth-order Runge-Kutta for the peridynamic model and provides proof of convergence order. Section 4 offers numerical examples demonstrating the MTS scheme's accuracy and convergence. Finally, Section 5 provides the conclusion.

\section{Fundamentals of bond-based peridynamics and its discretization}

This section provides a succinct overview of the linearized and nonlinearized versions of the bond-based peridynamics (PD) model, which are the main focuses of this study. Additionally, it briefly outlines the spatial discretization used to solve these models. The equation of motion for the bond-based peridynamics model is given below:
\begin{equation}\label{pd:e1}
\rho\dfrac{d^2\mathbf{u}}{dt^2}(\mathbf{x}_i,t)=\int_{\mathcal{H}_{\mathbf{x}}}  \mathbf{p}(\mathbf{u}(\mathbf{x}^{'}, t)-\mathbf{u}(\mathbf{x}, t), \mathbf{x}^{\prime}-\mathbf{x}) \mathrm{d} V_{\mathbf{x}^{\prime}}+\mathbf{b}(\mathbf{x}, t),
\end{equation}
where $\rho$ is the mass density, ${\mathcal{H}_{\mathbf{x}}}$ is the horizon which and usually taken as a disk or ball of radius $\delta$. $\mathbf{u}(\mathbf{x},t)$ and $\mathbf{b}(\mathbf{x},t)$ are the displacement vector and body force density fields, respectively. Let denote  $\bm{\xi}=\mathbf{x}^{\prime}-\mathbf{x}$ and  $\bm{\eta}=\mathbf{u}(\mathbf{x}^{\prime}, t)-\mathbf{u}(\mathbf{x}, t)$, then the pairwise (linear and nonlinear) force function $\mathbf{p}$ exerted by a particle $\mathbf{x}$ on another particle $\mathbf{x}^{\prime}$ can be expressed as follows:
\begin{itemize} 
\item linear form
 \begin{equation}\label{bond:e2}
    \mathbf{p}= \alpha\frac{\bm{\xi}\otimes\bm{\xi}}{\|\bm{\xi}\|^3}\bm{\eta},\\
\end{equation}
	\item nonlinear form
\begin{equation}\label{bond:e1}
\mathbf{p}=\alpha\frac{\|\bm{\xi}+\bm{\eta}\|-\|\bm{\xi}\|}{\|\bm{\xi}\|}\frac{\bm{\xi}+\bm{\eta}}{\|\bm{\xi}+\bm{\eta}\|}.    
\end{equation}  
\end{itemize}
Here $\alpha$ represents a scalar parameter introduced to ensure equivalence in energy between the peridynamic model and the classical elastic model,  as outlined below:
\begin{equation}\label{mu}
\alpha= \displaystyle{\begin{cases}\dfrac{9E}{\pi\delta^{3}d},& \text {2D},\vspace{0.5em}\\\dfrac{12E}{\pi\delta^{4}},&\text {3D},\\ \end{cases}}\hspace{0.5em} 
\end{equation}
where $E$ is the elastic modulus, $v$ is the Poisson's ratio, and $d$ is the thickness of the plate.

For discretizing \eqref{pd:e1} at a material point $\mathbf{x}_i$ for the linear bond-based peridynamic model, a meshfree method is employed for its simplicity in which the semi-discrete system can be rewritten as
\begin{equation}\label{mat:f1}
\rho\dfrac{d^2\mathbf{u}_i}{dt^2}=\sum_{\x_j\in\mathcal{H}_{x_i}} \alpha\frac{(\mathbf{x}_j-\mathbf{x}_i)\otimes(\mathbf{x}_j-\mathbf{x}_i)}{\|(\mathbf{x}_j-\mathbf{x}_i)\|^3}\left(\mathbf{u}_j-\mathbf{u}_i\right) V_{j}+\mathbf{b}(\mathbf{x}_i,t),
\end{equation} 
where $\mathbf{u}_i=\mathbf{u}(\mathbf{x}_i,t)$ and $V_j$ is the volume of the material point $\mathbf{x}_j$.

Similarly, for the nonlinear bond-based peridynamic model,  \eqref{pd:e1} can be discretized as
\begin{equation}\label{mat:f2}
\rho\dfrac{d^2\mathbf{u}_i}{dt^2}=\sum_{\x_j\in\mathcal{H}_{x_i}} 
\alpha\frac{\|\mathbf{x}_j+\mathbf{u}_j-\mathbf{x}_i+\mathbf{u}_i\|-\|\mathbf{x}_j-\mathbf{x}_i\|}{\|\mathbf{x}_j-\mathbf{x}_i\|}\frac{\mathbf{x}_j+\mathbf{u}_j-\mathbf{x}_i+\mathbf{u}_i}{\|\mathbf{x}_j+\mathbf{u}_j-\mathbf{x}_i+\mathbf{u}_i\|} V_{j}+\mathbf{b}(\mathbf{x}_i,t).
\end{equation}

By denoting $\mathbf{Q}=\{\mathbf{U}_j(t),\mathbf{x}_j\in\mathcal{H}_{\mathbf{x}_i}\}$, where $\mathbf{U}_i(t)=\left[\begin{matrix}\mathbf{u}(\mathbf{x}_i,t),\mathbf{v}(\mathbf{x}_i,t)\end{matrix}\right]^\mathrm{T}$, we can express the original semi-discrete system compactly as:
\begin{equation}\label{pd:dis1}
\dfrac{d{\mathbf{U}}_i}{dt}=\mathcal{L}_i\left(\mathbf{Q}\right),
\end{equation} 
where $\mathcal{L}$ is the spatial operator such that 
\begin{equation}\label{pd:dis2}
\mathcal{L}_i\left(\mathbf{Q}\right)=\left[\begin{matrix}
&\left(\sum\limits_{\x_j\in\mathcal{H}_{x_i}} \mathbf{p}(\mathbf{u}_j-\mathbf{u}_i, \mathbf{x}_j-\mathbf{x}_i) V_{j}+\mathbf{b}(\mathbf{x}_i,t)\right)/\rho\\
&\mathbf{v}(\mathbf{x}_i,t)
\end{matrix}\right],
\end{equation}
in which $\mathbf{v}=\dot{\mathbf{u}}$. 

\subsection{Temporal discretization schemes based on the Runge-Kutta methods}

This section presents a high-order temporal discretization scheme, without loss of generality and for clarity, using a uniformly-sized partition of the time interval $[0, T]$ into $N$ subintervals:
\begin{equation}
0=t_0\le t_1\le \cdots \le t_{N-1} \le t_N=T,\hspace{0.5em}\Delta t=t_{n+1} -t_{n},\hspace{0.5em} 0\leq n\leq N. 
\end{equation}

In order to discretize the system described by Eqs. \eqref{mat:f1} and \eqref{mat:f2} over time, we can use high-order Runge-Kutta (RK) methods that are commonly used to solve systems of ordinary differential equations. We define $\mathbf{R}^{(j)}(\mathbf{U}_i^n,\mathbf{Q}^n)$ as the solution at intermediate stage $j$ at time $t=t_n+c_k\Delta t$ for the general $r$-order $r$-stages explicit RK method applied to \eqref{pd:dis1}, where $j$, $k=1,2,\ldots,r$. Here, $\mathbf{U}_i^n$ represents the solution $\mathbf{U}(\mathbf{x}_i,t_n)$ and $\mathbf{Q}^n=\{\mathbf{U}_j^n,\ \mathbf{x}_j\in\mathcal{H}_{\mathbf{x}_i}\}$. The following equations are used for the RK method:

\begin{equation}\label{rkeq}
	\begin{aligned}
	\mathbf{R}^{(j)}(\mathbf{U}^n_i,\mathbf{Q}^n)&=\mathbf{U}^n_i+\Delta t \sum_{k=1}^r a_{jk} \mathcal{L}_i\left(\mathbf{R}^{(k-1)}, t_n+c_k \Delta t\right), \quad j=1, \ldots, r, \\
	\mathbf{U}^{n+1}_i&=\mathbf{U}^n_i+\Delta t \sum_{k=1}^r b_k \mathcal{L}_i\left(\mathbf{R}^{(k)}, t_n+c_k \Delta t\right),
	\end{aligned}
\end{equation}
where $\mathbf{R}^{(j)}=\{\mathbf{R}^{(j)}(\mathbf{U}^n_i,\mathbf{Q}^n),\mathbf{x}_j\in\mathcal{H}_{\mathbf{x}_i}\}$, and constants $c_j$, $a_{ij}$ and $b_k$ satisfy
\begin{equation}
c_j=\sum_{k=1}^r a_{jk}, \quad\sum_{k=1}^r b_{k}=1,\quad j=1, \ldots, r.
\end{equation}
The RK method can be represented compactly in the Butcher Tableau, as shown in Table \ref{RK}.

\begin{table}[!ht]\begin{center}	
		\caption{Butcher Tableau for the explicit $r$-order $r$-stages RK method}
		\vspace{0.15in}
		\setlength{\tabcolsep}{5mm}{
			\begin{tabular}{c|cccccc} \hline
				$0$ & &  & & \\
				$c_2$ & $a_{21}$ &  &  &\\
				$c_3$ & $a_{31}$ &$a_{32}$  &  &\\
				$\vdots$ &$\vdots$ &  & $\ddots$& \\
				$c_r$ & $a_{r1}$ & $a_{r2}$ & $\cdots$&$a_{r,r-1}$& \\
				\hline & $b_1$ &$b_2$ &$\cdots$&$b_r$
		\end{tabular}}\label{RK}
\end{center}\end{table}

A thorough description and analysis of the coefficients are given in \cite{RK}. In this case, we opt for a one-parameter family of a third-order, three-stages ($r$=3) explicit Runge-Kutta (RK) method, referred to as RK3, for the articulation of \eqref{pd:e1}, as illustrated in Table \ref{RK3}. To enhance the accuracy and stability of the solutions, an explicit fourth-order, four-stages ($r$=4) Runge-Kutta method, referred to as RK4, has also been adopted, as demonstrated in Table \ref{RK4}.

\begin{table}[!ht]\begin{center}	
		\caption{Butcher Tableau for the RK3}
		\vspace{0.15in}
		\setlength{\tabcolsep}{5mm}{
			\begin{tabular}{c|cccccc} \hline
				$0$ &  &  &  \\
				$2/3$ & $2/3$ &  & \\
				$2/3$ & $0$ & $2/3$ &  \\
				\hline & $1/4$ &$3/8$ &$3/8$
		\end{tabular}}\label{RK3}
\end{center}\end{table}

\begin{table}[!ht]\begin{center}	
		\caption{Butcher Tableau for the RK4}
		\vspace{0.15in}
		\setlength{\tabcolsep}{5mm}{
			\begin{tabular}{c|cccccc} \hline
				$0$ &  & &  &  \\
				$1/2$ & $1/2$ &  &  &  \\
				$1/2$ & $0$ & $1/2$ &  & \\
				$1$ & $0$ & $0$ & $1$ &  \\
				\hline & $1/6$ & $1/3$ & $1/3$ & $1/6$
		\end{tabular}}\label{RK4}
\end{center}\end{table}

High-order RK methods in PD temporal discretization offer superior accuracy and lower computational times compared to other approaches. However, the PD model often requires a finer time step size to handle crack-related issues, resulting in increased computational load. Striking a balance between accuracy and computational efficiency is key in PD simulations. A multi-time-step (MTS) scheme rooted in the RK method effectively solves this challenge, adapting fine time step sizes around cracks to maintain accuracy while reducing computational costs.

\subsection{Multi-time-step schemes}

This section presents the multi-time-step (MTS) scheme applied to a rectangular Cartesian domain denoted as $\Omega$. The domain $\Omega$ is divided into two subdomains: $\hat{\Omega}_C$ and $\hat{\Omega}_F$, that is, $\Omega = \hat{\Omega}_C \cup \hat{\Omega}_F$. Within $\hat{\Omega}_C$, a coarse time step size $\Delta t$ is assigned, while within $\hat{\Omega}_F$, a fine time step size $\Delta t_k$ is specified for $k = 0, \ldots, K$, where $K > 0$ is the refinement level. It is worth noting that the relationship $\Delta t = K\Delta t_k$ is enforced.

To address the nonlocal nature of the PD model and potential interactions between subdomains with different time step sizes, we further divide the entire domain $\Omega$ into four subdomains ($\Omega_F, \Omega_{FI}, \Omega_{cI}, \Omega_C$) as depicted in Fig. \ref{Fig:2}. Subdomains $\Omega_{FI}$ and $\Omega_{CI}$ act as internal boundary layers (BL). 
\\

$\bullet$ $\Omega_{F}=\hspace{0.4em}\left\{\mathbf{x}_i:\mathbf{x}_i\in\hat{\Omega}_F,\hspace{0.5em}\left( \hat{\Omega}_F\cap\mathcal{H}_{\mathbf{x}_i}\right)=\mathcal{H}_{\mathbf{x}_i}\right\};$\vspace{0.3em}

$\bullet$ $\Omega_{FI}=\left\{\mathbf{x}_i:\mathbf{x}_i\in\hat{\Omega}_F,\hspace{0.5em} \left(\hat{\Omega}_F\cap\mathcal{H}_{\mathbf{x}_i}\right)\neq\mathcal{H}_{\mathbf{x}_i}\right\}$;\vspace{0.3em}

$\bullet$ $\Omega_{CI}=\left\{\mathbf{x}_i:\mathbf{x}_i\in\hat{\Omega}_C,\hspace{0.5em} \left(\hat{\Omega}_C\cap\mathcal{H}_{\mathbf{x}_i}\right)\neq\mathcal{H}_{\mathbf{x}_i}\right\}$;\vspace{0.3em}

$\bullet$ $\Omega_{C}=\hspace{0.4em}\left\{\mathbf{x}_i:\mathbf{x}_i\in\hat{\Omega}_C,\hspace{0.5em} \left(\hat{\Omega}_C\cap\mathcal{H}_{\mathbf{x}_i}\right)=\mathcal{H}_{\mathbf{x}_i}\right\}$.\vspace{0.3em}

\begin{figure}[!ht]
	\centering            
	\includegraphics[scale=0.45]{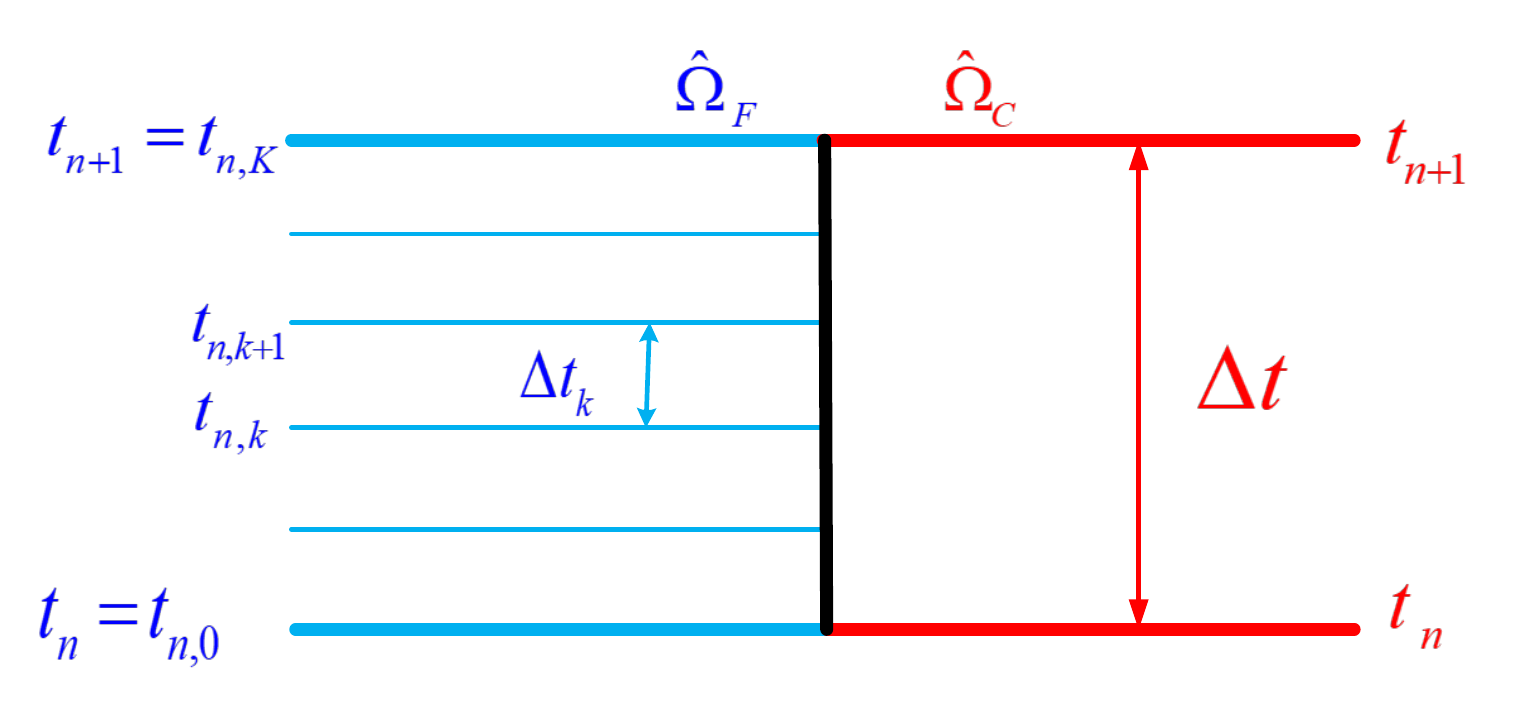}  
	\caption{Different time step size for the subdomains $\hat{\Omega}_{F}$ (fine time step size $\Delta t_k$) and $\hat{\Omega}_{C}$ (coarse time step size $\Delta t$).
	} 
	\label{Fig:CF}
\end{figure}

\begin{figure}[!ht]
	\centering            
	\includegraphics[scale=0.5]{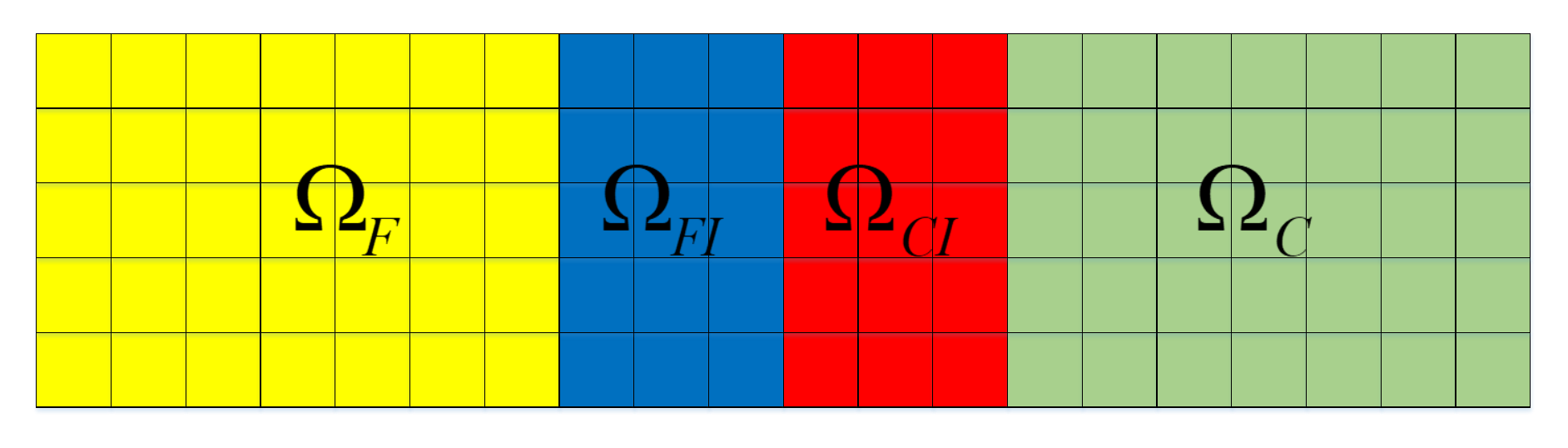}  
	\caption{Four subdomains of the physical domain $\Omega$. The subdomain $\Omega_{F}$, BL subdomain $\Omega_{FI}$,  BL subdomain $\Omega_{CI}$, and subdomain$\Omega_{C}$ are colored in yellow, blue, red, and green, respectively. Here  $\hat{\Omega}_C=\Omega_{C}\cup\Omega_{CI}$ and $\hat{\Omega}_F=\Omega_{F}\cup\Omega_{FI}$.
	} 
	\label{Fig:2}
\end{figure}

As outlined below, we shall use different numerical methods to compute material points in different subdomains based on their distinct properties.
\begin{itemize}
 \item Material points in subdomains $\Omega_{C}\cup\Omega_{F}$, which are not affected by subdomain $\hat{\Omega}_F$ with a finer time step size of $\Delta t_k$, are computed using the RK method with their respective time step size, that is, $\Delta t$ for $\Omega_C$ and $\Delta t_k$ for $\Omega_F$. 
\item  Material points in BL subdomain $\Omega_{CI}$ have a time step size of $\Delta t$ but are influenced by subdomain $\hat{\Omega}_F$ that has a time step size of $\Delta t_k$. For this subdomain, we incorporate {\em correction term} as shown in \eqref{lrkc32}. 
\item  Material points in BL subdomain $\Omega_{FI}$ have a time step size of $\Delta t_k$ and are affected by subdomain $\hat{\Omega}_{C}$ that has a time step size $\Delta t$. However, since the solution advances from $t_{n,k}$ to $t_{n,k+1}$ with a finer time step size $\Delta t_k$, the intermediate solutions in $\Omega_{CI}$ are unavailable. To compensate for this, we use {\em high-order interpolation polynomials} to estimate these missing intermediate solutions as depicted in Fig. \ref{Fig:CF}.
\end{itemize}
Therefore, integrating methods from different subdomains is essential for constructing the MTS scheme. In the following sections, we will provide specific advanced stages for third- and fourth-order MTS schemes and prove their accuracy.

\subsubsection{The third and fourth order multi-time-step scheme}

In this section, we introduce the third-order and fourth-order MTS schemes as solutions to address the challenge of varying time step sizes. To begin the MTS schemes, the solutions at $t_0$ and $t_1$ are defined as follows:
\begin{equation}
	\mathbf{U}^0_i=(\mathbf{u}_0,\mathbf{v}^0)^{\mathrm{T}},\hspace{0.5em} \mathbf{U}^1_i=\mathbf{U}^0_i++\Delta t \sum_{k=1}^r b_k \mathcal{L}_i\left(\mathbf{R}^{(j)}, t_n+c_k \Delta t\right),
\end{equation}
where $\mathbf{u}_0$ and $\mathbf{v}^0$ is the displacement and velocity at $t_0$, respectively. $\mathbf{R}^{(j)}$ is obtained by the \eqref{rkeq} when $n=0$. Then the MTS scheme consists of three steps: first, advance $\mathbf{U}^i_n$ from $t_n$ to $t_{n+1}$ ($n\geq 2$) for the material points on $\hat{\Omega}_C$, next utilize their solutions to construct an interpolation polynomial and update the intermediate solution at the fine time step sizes, and finally advance on the subdomain $\hat{\Omega}_{F}$.
\begin{itemize}
\item[1)] Advancing from $t_n$ to $t_{n+1}$ ($n\geq 2$) on the  subdomain  $\hat{\Omega}_C$ \\

Due to the influence of material points on the BL subdomain $\Omega_{FI}$, the advancement from $t_n$ to $t_{n+1}$ of material points in the BL subdomain $\Omega_{CI}$ requires the inclusion of a {\em correction term} in the expression for $\mathbf{R}^{(j)}_i(\mathbf{U}^n_i,\mathbf{Q}^n)$, where $j=1,\dots,r$. Let us denote $\overline{\Omega}_{C}=\hat{\Omega}_C\cup\Omega_{FI}$, $\mathbf{Q}^i_C=\mathbf{U}^{i}|_{\overline{\Omega}_{C}}$, and $\mathbf{Q}^{(i)}_C=\mathbf{U}^{(i)}|_{\overline{\Omega}_{C}}$. For a $r$-order $r$-stages RK method on the subdomain $\overline{\Omega}_{C}$, the following expression can be obtained for the value at the intermediate stages $\mathbf{U}^{(j)}_i$:
	\begin{equation}\label{lrkc32}
	\mathbf{U}^{(j)}_i=\begin{cases}&\mathbf{R}^{(j)}_i(\mathbf{U}^n_i,\mathbf{Q}^n_C)\vspace{0.5em},\mathbf{x}_i\in\Omega_{CI},\\
	&\mathbf{U}^n_i+\Delta t \sum\limits_{k=1}^r a_{jk} \mathcal{L}_i\left(\mathbf{D}_C^{(k-1)}, t_n+c_k \Delta t\right),\\
	\end{cases}
	\end{equation}
	where $\mathbf{D}_C^{(j)}=\{\mathbf{R}^{(j)}(\mathcal{L}_i(\mathbf{Q}^n_C),\mathbf{G}^n),\hspace{0.5em}\mathbf{x}_j\in\mathcal{H}_{\mathbf{x}_i}\}$, $\mathbf{G}^n_C=\{\mathcal{L}_j(\mathbf{Q}^n_C),\hspace{0.5em}\mathbf{x}_j\in\mathcal{H}_{\mathbf{x}_i}\}$. For the derivatives in $\mathbf{D}^{(j)}_C$, we  use the backward difference scheme to estimate, as follows:
	\begin{equation}\label{sd}
	\begin{aligned}
	&\mathcal{L}^{''}(\mathbf{Q}^n_C)=	\dfrac{\mathcal{L}_i\left(\mathbf{Q}^n_C\right)-\mathcal{L}_i\left(\mathbf{Q}^{n-1}_C\right)}{\Delta t},\\
	&\mathcal{L}^{'''}(\mathbf{Q}^n_C)=\dfrac{\mathcal{L}_i\left(\mathbf{Q}^n_C\right)-\mathcal{L}_i\left(\mathbf{Q}^{n-1}_C\right)}{\Delta t^2}-\dfrac{\mathcal{L}_i\left(\mathbf{Q}^{n-1}_C\right)-\mathcal{L}_i\left(\mathbf{Q}^{n-2}_C\right)}{\Delta t^2},
	\end{aligned}
	\end{equation}
	then the value $\mathbf{U}^{n+1}_i$ for $\mathbf{x}_i\in\hat{\Omega}_{C}$  can be obtained by
	\begin{equation}
	\mathbf{U}^{n+1}_i=\mathbf{U}^n_i+
	\begin{cases}
	&\dfrac{\Delta t}{4} \left(\mathcal{L}_i\left(\mathbf{Q}^n_C\right)+\mathcal{L}_i\left(\mathbf{Q}^{(1)}_C\right)+\mathcal{L}_i\left(\mathbf{Q}^{(2)}_C\right)\right),\hspace{0.5em}\text{$r$=3},\vspace{0.5em}\\
	&\dfrac{\Delta t}{6} \Big(\mathcal{L}_i\left(\mathbf{Q}^n_C\right)+2\mathcal{L}_i\left(\mathbf{Q}^{(1)}_C\right)+2\mathcal{L}_i\left(\mathbf{Q}^{(2)}_C\right)+2\mathcal{L}_i\left(\mathbf{Q}^{(3)}_C\right)\Big),\hspace{0.5em}\text{$r$=4}.
	\end{cases}			
	\end{equation}
	Therefore, the value $\mathbf{U}_i^n$ is advanced from $t_n$ to $t_{n+1}$ on the subdomain $\hat{\Omega}_C$.

	\item[2)] Update the intermediate solution $\mathbf{U}^{n,k}_i$ on the BL subdomain $\Omega_{CI}$ \\
	
To make progress within the subdomain $\hat{\Omega}_F$, the computation of intermediate temporal values $\mathbf{U}^{n,k}_i$ for $\mathbf{x}_i\in\Omega_{IC}$ is necessary. Here, the index $k$ ranges from 1 to $K$, representing the number of intermediate time levels. To predict $\mathbf{U}^{n,k}_i$, an interpolation technique is employed using an interpolating polynomial $\mathcal{I}_i(t)$. It is required that this polynomial satisfies specific properties for a $r$-order $r$-stages Runge-Kutta explicit scheme, as follows.
	
	\begin{equation}
	\mathcal{I}_i(t_{n,k}) = \mathbf{U}^{n,k}_i + O(\Delta t^{r+1}).
	\end{equation}
	
In this context, it is assumed that the operator $\mathcal{L}$ is sufficiently smooth, which allows us to employ the Taylor expansion for constructing the interpolation polynomial:
	\begin{equation}\label{Lte}
	\begin{aligned}
	\mathbf{U}^{n,k}_i = \mathbf{U}^n_i + \mathcal{L}_i\left(\mathbf{Q}^n_C\right)(t_{n,k}-t_n) + \cdots+ \mathcal{L}_i^{(r-1)}\left(\mathbf{Q}^n_C\right)\frac{(t_{n,k}-t_n)^3}{r!} + O(\Delta t^{r+1}).
	\end{aligned}
	\end{equation}
		
    Let's define $\mathbf{A}^r$ as an $(r-1) \times (r-1)$ matrix, and $\mathbf{f}^r$ and $\mathbf{d}^r$ as $(r-1)$-dimensional vectors, where $d_i^r = \mathcal{L}^{(i-1)}(\mathbf{Q}^n_C)$. Then, the derivative of $\mathcal{L}_i$ can be estimated as:
    \begin{equation}
    \mathbf{f}^r = \mathbf{A}^r\mathbf{d}^r + O(\Delta t^{r+1}).
    \end{equation}   
    We remark that the entries $f^r_i$ in the vector $\mathbf{f}^r$ and $a^r_{ij}$ in the matrix $\mathbf{A}^r$ can be obtained easily using the Taylor series expansion. 
    
    Specifically, for the fourth-order scheme ($r$=4), one has
	\begin{equation}
		\mathbf{f}^4=
		\left[
		\begin{array}{ccc}
		&\dfrac{\mathbf{U}^{n+1}_i-\mathbf{U}^n_i-\Delta t \mathcal{L}_i\left(\mathbf{Q}^n_C\right)}{\Delta t^{2}}\vspace{0.2em}\\
		&\dfrac{\mathcal{L}_i\left(\mathbf{Q}^n_C\right)-\mathcal{L}_i\left(\mathbf{Q}^{n-1}_C\right)}{\Delta t}\vspace{0.2em}\\
		&\dfrac{\mathcal{L}_i\left(\mathbf{Q}^{n-1}_C\right)-\mathcal{L}_i\left(\mathbf{Q}^{n-2}_C\right)}{\Delta t}\vspace{0.2em}\\
		\end{array}\right],\hspace{0.5em}\mathbf{A}^4=
		\left[
		\begin{array}{rrr}
		\dfrac{1}{2}&\dfrac{\Delta t}{6}&\dfrac{\Delta t^2}{24}\vspace{0.3em}\\
		1&-\dfrac{\Delta t}{2}&\dfrac{\Delta t^2}{6}\vspace{0.3em}\\
		1&-\dfrac{3\Delta t}{2}&\dfrac{7\Delta t^2}{6}\\
		\end{array}\right],
		\end{equation}
then $\mathbf{d}^4$ can be obtain by $\mathbf{d}^r=(\mathbf{A}^r)^{-1}\mathbf{f}^r+O(\Delta t^{r+1})$, where
\begin{equation}
(\mathbf{A}^4)^{-1}=
\left[
\begin{array}{ccc}
\dfrac{8}{9}&\dfrac{37}{54}&-\dfrac{7}{54}\vspace{0.3em}\\
\dfrac{8}{3\Delta t}&-\dfrac{13}{9\Delta t}&\dfrac{1}{9\Delta t}\vspace{0.3em}\\
\dfrac{8}{3\Delta t^2}&-\dfrac{22}{9\Delta t^2}&\dfrac{10}{9\Delta t^2}\\
\end{array}\right].
\end{equation}

Substituting results into \eqref{Lte}, an interpolation polynomial can be obtained as follows when $t_n \leq t \leq t_{n+1}$:
\begin{equation}\label{I2}
\begin{aligned}
\mathcal{I}_i(t)= & \mathbf{U}^n_i+\left(t-t_n\right) \mathcal{L}_i\left(\mathbf{Q}^{n}_C\right)+\sum_{j=2}^{r-1} \dfrac{\left(t-t_n\right)^{j}}{j!}\beta_{j-1},
\end{aligned}
\end{equation}
\normalsize
with
\begin{equation}
	\beta_j=\sum\limits_{i=1}^{r-1}\gamma_{ji}f_i,
\end{equation}
where $\gamma_{ij}$ is the entry in $(\mathbf{A}^r)^{-1}$.
Finally, one has $\mathbf{U}^{n,k}_i = \mathcal{I}_i(t_{n,k}), \mathbf{x}_i\in\Omega_{CI}$.

\begin{remark}
For the third-order ($r$=3) scheme, one set 
\begin{equation}\label{relam}
a^{3}_{ij}=a^4_{ij},\hspace{0.5em}f^{3}_i=f^4_i.
\end{equation}
then 
$\mathbf{d}^3$ can be obtained in the same manner as $\mathbf{d}^4$, where
\begin{equation}
(\mathbf{A}^3)^{-1}=
\left[
\begin{array}{ccc}
\dfrac{6}{5}&\dfrac{2}{5}\vspace{0.5em}\\
\dfrac{12}{5\Delta t}&-\dfrac{-6}{5\Delta t}\\
\end{array}\right],
\end{equation}
and $\mathcal{I}_i(t)$ is computed by \eqref{I2}.
\end{remark}

\item[3)] Advancing from $t_n$ to $t_{n+1}$ on the subdomain $\hat{\Omega}_F$ \\

By constructing $\mathbf{U}^{n,k}$ through interpolation polynomials, the expression $\mathbf{R}^{(j)}_i(\mathbf{U}^n_i,\mathbf{Q}^n)$, where $j=1,\dots,r$ and $\mathbf{x}_i\in\Omega_{IC}$, can be rewritten to advance the material points on $\hat{\Omega}_F$. Let $\overline{\Omega}_{F}=\Omega_{CI}\cup\hat{\Omega}_F$, $\mathbf{Q}^i_F=\mathbf{U}^{i}|_{\overline{\Omega}_{F}}$, and $\mathbf{Q}^{(i)}_F=\mathbf{U}^{(i)}|_{\overline{\Omega}_{F}}$. The solution at the intermediate stages is then computed using the Runge-Kutta method for $k=1,2,\dots, K$ ($K > 1$ is the refinement level):
		\begin{equation}\label{lrk31}
		\hspace{-9.5em}
		\overline{\mathbf{U}}^{(j,k)}_i=
		\begin{cases}
		&\mathbf{R}^{(j)}_i(\mathbf{U}^{n,k}_i,\mathbf{Q}^{n,k}_F),\hspace{0.5em} \mathbf{x}_i\in\hat{\Omega}_F,\vspace{0.5em}\\
		&\mathbf{U}^n_i+\Delta t \sum\limits_{k=1}^r a_{jk} \mathcal{L}_i\left(\mathbf{D}_F^{(k-1)}, t_n+c_k \Delta t\right),\hspace{0.5em}\mathbf{x}_i\in\Omega_{CI},\\
		\end{cases}
		\end{equation}
where $\mathbf{R}^{(j)}_i$ is obtained by Tables \ref{RK3} and \ref{RK4} for the RK3 and RK4 schemes, respectively, and where $\mathbf{D}_F^{(j)}=\{\mathbf{R}^{(j)}_i(\mathcal{I}_i(t_{n,k}),\mathbf{I}),\hspace{0.5em}\mathbf{x}_j\in\mathcal{H}_{\mathbf{x}_i}\}$, $\mathbf{I}=\{\mathcal{I}_j(t_{n,k}),\hspace{0.5em}\mathbf{x}_j\in\mathcal{H}_{\mathbf{x}_i}\}$. Then the value $\mathbf{U}^n_i$ can be advanced from $t_n$ to $t_{n+1}$ by
		\begin{equation}\label{lrk33}
		\mathbf{U}^{n,k+1}_i=\mathbf{U}^{n,k}_i+
		\begin{cases}
		\dfrac{\Delta t_{k}}{4}\left(\mathcal{L}_i\left(\mathbf{Q}^{n,k}_F\right)+\frac{3}{2} \mathcal{L}_i\left(\mathbf{Q}^{(1,k)}_F\right)+\dfrac{3}{2}\right.\left. \mathcal{L}_i\left(\mathbf{Q}^{(2,k)}_F\right)\right), \hspace{0.5em} &\text{$r$=3},\vspace{0.5em}\\
		\dfrac{\Delta t_k}{6}\Bigg(\mathcal{L}_i\left(\mathbf{Q}^{n,k}_F\right)+2 \mathcal{L}_i\left(\mathbf{Q}^{(1,k)}_F\right)+2 \mathcal{L}_i\left(\mathbf{Q}^{(2,k)}_F\right)+\mathcal{L}_i\left(\mathbf{Q}^{(3,k)}_F\right)\Bigg),  \hspace{0.5em} &\text{$r$=4},\\
		\end{cases}
		\end{equation}
\end{itemize}
as a result of the aforementioned computations, the value $\mathbf{U}^{n+1}_i$ is obtained for all material points $\mathbf{x}_i\in\Omega$.
\subsubsection{Convergence analysis}
The MTS and UPD schemes exhibit the same convergence behavior in smooth cases. The third- and fourth-order MTS schemes follow the theorem below: 
 \begin{theorem} 
If the operator $\mathcal{L}$ is $C^5$, the global error of the r-order MTS scheme under the uniform time step size $\Delta t$ is $O(\Delta t^{r})$, where $\mathbf{U}(\mathbf{x},t)$ is the exact solution of the semi-discrete system in \eqref{pd:dis1}. This holds for $r=3,4$.
 \end{theorem}

 \begin{proof}
Since the third-order scheme is a special form of the fourth-order scheme, we first provide proof of the fourth-order scheme. Without loss of generality, we assume $\mathbf{e}^{m}_i=0$ and $\mathbf{e}^{m,n}_i=0$ for  $m=1,2,\cdots, n$ and $n=1,2,\cdots, k$.
\begin{itemize}
\item[i)] First, we prove that the local error for the material points on the subdomain $\hat{\Omega}_C$ is $O(\Delta t^5)$.  Based on the construction process of the MTS scheme, the following expression is obtained:
 	\begin{align}\label{tyrk4}
 	\mathbf{U}^{n+1}_i= &\mathbf{U}^{n+1}_i+\dfrac{\Delta t}{6} \Big(\mathcal{L}_i\left(\mathbf{Q}^n_C\right)+2\mathcal{L}_i\left(\mathbf{Q}^{(1)}_C\right)+2\mathcal{L}_i\left(\mathbf{Q}^{(2)}_C\right)+2\mathcal{L}_i\left(\mathbf{Q}^{(3)}_C\right)\Big)\notag\\=&\mathbf{U}^{n+1}_i+\Delta t\mathcal{L}(\mathbf{Q}^n_C)+\frac{\Delta t^2}{2}\mathcal{L}^{'}(\mathbf{Q}^n_C)+\frac{\Delta t^2}{3!}\mathcal{L}^{''}(\mathbf{Q}^n_C)+\frac{\Delta t^4}{4!}\mathcal{L}^{'''}(\mathbf{Q}^n_C)+O(\Delta t^5)\notag\\
=&\hat{\mathbf{U}}^{n}_i+\Delta t \hat{\mathbf{U}}_i^{'}(t_n) +\frac{\Delta t^2}{2} \hat{\mathbf{U}}_i^{''}(t_n)+\frac{\Delta t^3}{3!} \hat{\mathbf{U}}_i^{'''}(t_n)+\frac{\Delta t^4}{4!} \hat{\mathbf{U}}_i^{''''}(t_n)+O(\Delta t^5)\\=&\hat{\mathbf{U}}_i(t_{n+1})+O(\Delta t^5),\hspace{0.5em}\mathbf{x}_i\in\Omega_{CI}.\notag
 	\end{align}
 	Thus the local error $\mathbf{e}_i^{n+1}$ for the material points on the BL subdomain $\Omega_{CI}$ is $O(\Delta t^5)$. It can be observed that the value of the material points on the subdomain $\Omega_{C}$ is obtained by RK4; thus, the local error is also $O(\Delta t^5)$.
\item[ii)] With the error obtained by the interpolation polynomial, the local error on the subdomain $\hat{\Omega}_F$ can be demonstrated as follows:
\begin{equation}
\mathcal{I}_i(t_{n,k})=\hat{\mathbf{U}}_i(t_{n,k})+O(\Delta t_k^5).
\end{equation}
Substituting into \eqref{lrk31}, we have
\begin{equation}
\overline{\mathbf{U}}^{(j,k)}_i=
\begin{cases}
&\mathbf{R}^{(j)}_i(\mathbf{U}^{n,k}_i,\mathbf{Q}^{n,k}_F),\hspace{0.5em} \mathbf{x}_i\in\hat{\Omega}_F,\vspace{0.5em}\\
&\mathbf{R}^{(j)}_i(\mathbf{U}^{n,k}_i,\hat{\mathbf{Q}}^{'}(t_{n,k}))+O(\Delta t_k^5),\hspace{0.5em}\mathbf{x}_i\in\Omega_{CI},\\
\end{cases}
\end{equation}
where $\hat{\mathbf{Q}}^{'}(t_{n,k})=\{\hat{\mathbf{U}}^{'}(t_{n,k}),\hspace{0.5em}\mathbf{x}_j\in\mathcal{H}_{\mathbf{x}_i}\}$.
By Taylor expansion, we can obtain follows 
\begin{equation}
	\mathbf{U}^{n,k+1}_i=\hat{\mathbf{U}}_i(t_{n,k+1})+O(\Delta t_k^5), \hspace{0.5em}\mathbf{x}_i\in\hat{\Omega}_F.
\end{equation}
Then the local error of the material points on the subdomain $\overline{\Omega}_F$ is $O(\Delta t_k^5) = O(\Delta t^5)$.

\item[iii)] The global error of the material points in the domain $\Omega$ can be obtained by considering the local error. Let $\mathbf{E}^{n+1}$ represent the value at $t_{n+1}$ of the local error committed at $t_{n,k}$ and propagated forward by the numerical scheme. The local error $\mathbf{e}^{n+1}_i$ can be expressed as the sum of $\mathbf{E}^{n,k}_i$ for $k=1,2,\dots,K$. Thus, the accumulated error can be estimated as follows:
\begin{equation}
\left|\mathbf{e}^{n+1}_i\right| \leqslant \sum_{k=1}^K\left|\mathbf{E}^{n, k}_i\right|.
\end{equation}
According to Theorem I.10.2 in \cite{AR}, it can be shown that:
\begin{equation}
\left|\mathbf{E}^{n, k}_i\right| \leqslant \mathbf{e}^{B\left(t_{n+1}-t_{n, k}\right)}_i\left|\mathbf{e}^{n, k}_i\right| \leqslant \mathbf{e}^{B\Delta t}_i\left|\mathbf{e}^{n, k}_i\right|,
\end{equation}
where $B$ is an upper bound on the Jacobian of \eqref{pd:dis1}. Since $|\mathbf{e}^{n,k}_i| \le C_2 \Delta t^4$, the estimate becomes:
\begin{equation}
\left|\mathbf{e}^{n+1}_i\right| \leqslant K\mathbf{e}^{B\Delta t}_i\left|\mathbf{e}^{n, k}_i\right| \leqslant K\left(1+\Delta t B+\frac{(\Delta t B)^2}{2}+\cdots\right) C_2 \Delta t^5 \leqslant C_3 \Delta t^5,
\end{equation}
here $C_2$ and $C_3$ are some constants.
Combining this with Theorem II.3.4 in \cite{AR}, we obtain the expression for the global error $\mathbf{E}_i$:
\begin{equation}
|\mathbf{E}_i| \leqslant  \frac{C_2}{K}\left(\mathbf{e}^{BT}-1\right)\Delta t^4.
\end{equation}
Similarly, one can also show that the global error of the third-order MTS scheme is $O(\Delta t^3)$. Thus, the theorem is proved.
\end{itemize}
 \end{proof}
 
 \section{Numerical results}
In this section, we present three examples to demonstrate the accuracy and convergence of the MTS scheme for bond-based peridynamic models on various geometries. All experiments were implemented in Fortran$^\circledR$ and conducted on a workstation with an Intel Xeon Gold 6240 CPU operating at 2.6GHz, with 2048GB of installed memory.
\subsection{Plate under  transverse loading}
In this section, we will consider a 2D linear bond-based peridynamic model applied to a plate under uniaxial tension. The plate has dimensions of $W = 1 \mathrm{~m}$ (width), $L = 0.5 \mathrm{~m}$ (length), and $d = 0.01 \mathrm{~m}$ (thickness). The geometry of the plate is illustrated in Figure \ref{Fig:3}.

\begin{figure}[!ht]
	\centering  
	\subfigure[]  	
	{
		\begin{minipage}{7cm}
			\centering          
			\includegraphics[scale=0.38]{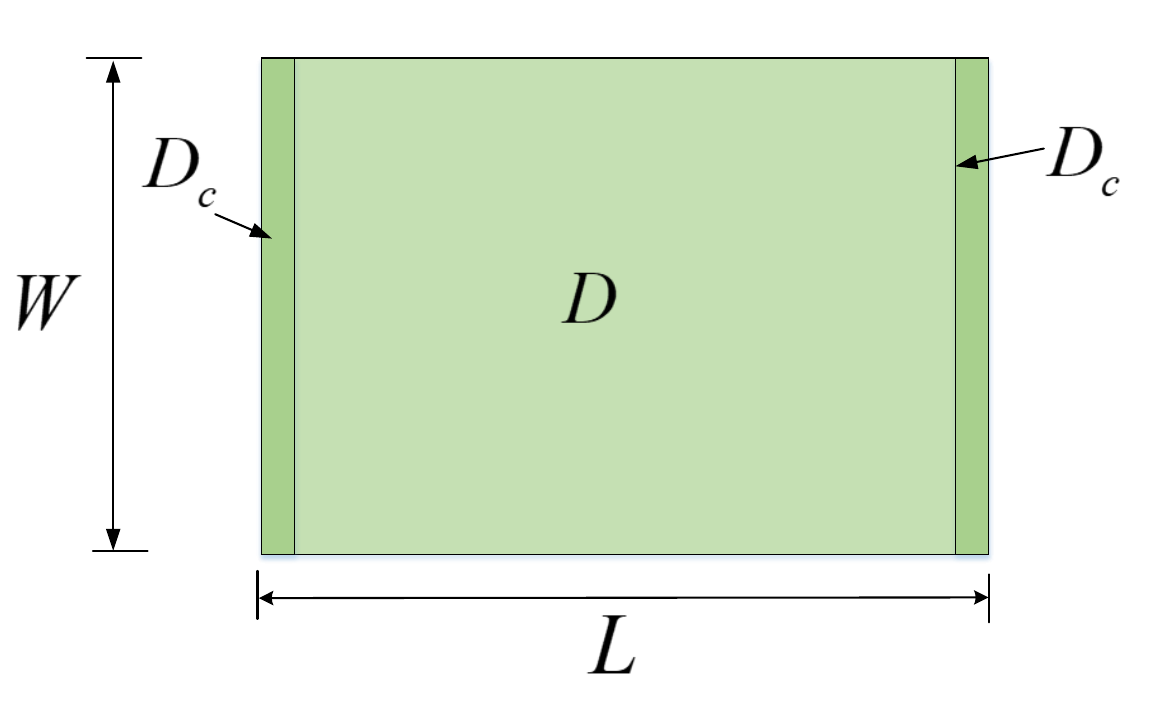}  
		\end{minipage}
	}
	\subfigure[]	
	{
		\begin{minipage}{7cm}
			\centering      
			\includegraphics[scale=0.38]{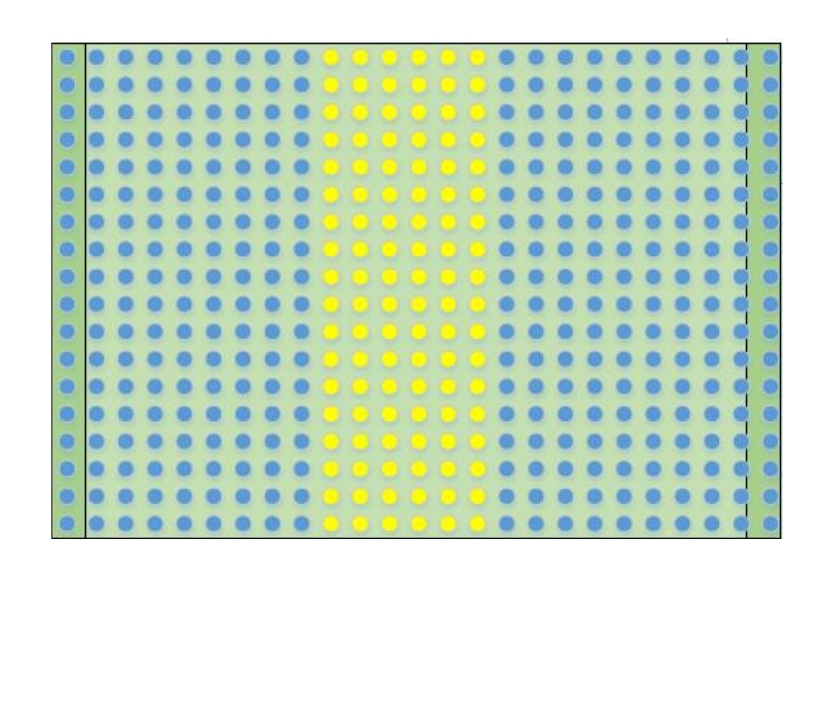}   
		\end{minipage}
	}
	
	\caption{Geometry of a 2D plate  under  transverse loading and its discretization.} 
	\label{Fig:3}
\end{figure}

The material properties are $E=1.92\times 10^{11}$ Mpa (elastic modulus), $\nu=1/3$ (Poisson's ratio) and $\rho=8000 \mathrm{~kg}/\mathrm{m}^3$ (density). The external loading $b_p=p_0W/d$ is applied to boundary layer $\mathcal{D}_c$, where the uniaxial tension loading $p_0 = 200 \mathrm{~Mpa}$, and the width of $\mathcal{D}_c$ is $d$. 
The horizon size $\delta$ is chosen as $0.03 \mathrm{~m}$.

For the discretization of the model, a mesh with a size of $100 \times 50$ is used. The yellow subdomain in Fig. \ref{Fig:3}(b) is computed using the MTS schemes with $\Delta t_k, k = 1,\ldots,K$ where $K$ is the refinement level, while the displacements in the other areas are obtained using $\Delta t$. To compare the performance of the MTS scheme and undecomposed PD (UPD) scheme, the time step sizes are refined gradually from $1.0 \times 10^{-5}\hspace{0.5em}\mathrm{s}$ to $0.125 \times 10^{-5}\hspace{0.5em}\mathrm{s}$, and the total simulation time considered is $4.0 \times 10^{-4}\hspace{0.5em}\mathrm{s}$. $L^2$ absolute errors $\|u_y-u_y^h\|_2$ is used to assess the accuracy of the schemes, where $u_y$ represents the displacement in the $y-$direction, $u_y^h$ represents the exact solution obtained using the displacement in the $y-$direction from the UPD scheme when the time step size is $\Delta t =\sqrt{4/5}\times10^{-7} \hspace{0.5em}\mathrm{s}$. 
 
\begin{table}[!ht]\begin{center}
		\caption{The error and order (CR) obtain by MTS scheme based on RK3 (MTS3) and RK4 (MTS4) with different $\Delta t$ and different refinement level K on a 2D model}
		\vspace{0.15in}
		\setlength{\tabcolsep}{2mm}{
			\begin{tabular}{|c|c|l|c|c|l|c|} \hline
				\multicolumn{2}{|c|}{}&\multicolumn{2}{c|}{MTS3}&\multicolumn{2}{c|}{MTS4}\\ \hline
				$\Delta t$ (s)& $K$ & error & CR& error & CR\\ \hline				
				\multirow{4}*{\makecell[c]{\hspace{0.75em}$1.0$E$-5$}} 
				&1&$3.28$E$-10$&-&$7.62$E$-11$&- \\
				&2&$2.54$E$-10$&-&$4.45$E$-11$&-\\
				&4&$2.41$E$-10$&-&$4.37$E$-11$&-\\
				&8&$2.40$E$-10$&-&$4.35$E$-11$&-\\\hline
				\multirow{4}*{\makecell[c]{\hspace{0.75em}$0.5$E$-5$}} 
				&1&$8.87$E$-11$&3.07&$5.45$E$-12$&3.91\\
				&2&$5.48$E$-11$&3.07&$3.16$E$-12$&3.91\\
				&4&$5.07$E$-11$&3.08&$3.04$E$-12$&3.93\\
				&8&$5.02$E$-11$&3.09&$3.03$E$-12$&3.93\\\hline
				\multirow{4}*{\makecell[c]{\hspace{0.25em}$0.25$E$-5$}} 
				&1&$1.28$E$-11$&3.03&$3.46$E$-13$&3.96\\
				&2&$7.66$E$-12$&3.04&$1.99$E$-13$&3.98\\
				&4&$7.10$E$-12$&3.04&$1.92$E$-13$&3.97\\
				&8&$7.03$E$-12$&3.04&$1.91$E$-13$&3.97\\\hline
				\multirow{4}*{\makecell[c]{$0.125$E$-5$}} 
				&1&$1.63$E$-12$&3.00&$2.17$E$-14$&3.98\\
				&2&$9.71$E$-13$&3.01&$1.25$E$-14$&3.98\\
				&4&$8.99$E$-13$&3.00&$1.20$E$-14$&3.99\\
				&8&$8.91$E$-13$&3.00&$1.20$E$-14$&3.99\\\hline
		\end{tabular}}\label{tab:1}
\end{center}\end{table}
Table \ref{tab:1} provides the error and order for the two-dimensional problem at different time step sizes. The comparison between the UPD scheme ($K = 1$) and the MTS scheme ($K > 1$) reveals that the relative error decreases as the refinement level $K$ increases. This indicates that refining $\Delta t$ by increasing $K$ leads to improved accuracy. For example, when refining $K$ from two to eight with $\Delta t= 0.5 \times 10^{-5}\hspace{0.5em}\mathrm{s}$, the error reduces from $8.87 \times 10^{-11}$ to $5.02 \times 10^{-11}$, highlighting the accuracy of the MTS  scheme.

Furthermore, an interesting observation is the preservation of the order of the UPD scheme in the MTS method. This means that the order of the MTS scheme, whether based on the RK3 or RK4 methods, remains consistent even when the time step size $\Delta t$ varies. This finding validates the efficacy of the analysis of numerical schemes and demonstrates that the MTS scheme maintains its expected order of accuracy across different time step sizes.
 
\subsection{Block under  transverse loading}
In this section, we conducted a PD simulation of a 3D block subjected to quasi-static transverse loading. The block has dimensions of $1.0\mathrm{~m}$ in length, $0.3\mathrm{~m}$ in width, and $0.3\mathrm{~m}$ in thickness, as shown in Fig. \ref{Fig:4} (a). The horizon size $\delta$ was chosen as $0.03\mathrm{~m}$. The external loading of $b_p=p_0W/d$ was applied to the area $D_c$, where the thickness $d$ of $D_c$ is $0.01\mathrm{~m}$, and the value of $p_0$ is $200\mathrm{~MPa}$. The material properties used in the simulation are an elastic modulus $E$ of $2.0\times 10^5\mathrm{~MPa}$ and a Poisson's ratio $\nu$ of $1/4$.
\begin{figure}[!ht]
	\centering  
	\subfigure[]  	
	{
		\begin{minipage}{7.5cm}
			\centering          
			\includegraphics[scale=0.7]{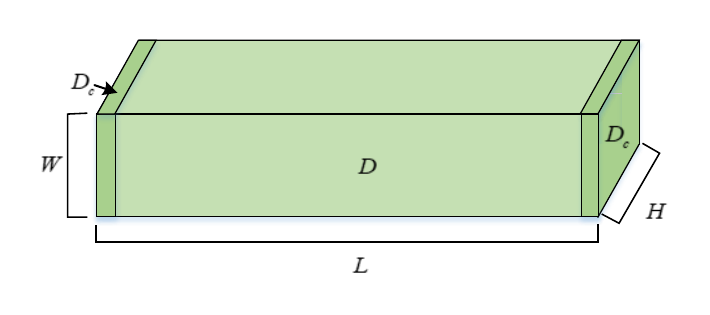}  
		\end{minipage}
	}
	\subfigure[]	
	{
		\begin{minipage}{7.5cm}
			\centering      
			\includegraphics[scale=0.75]{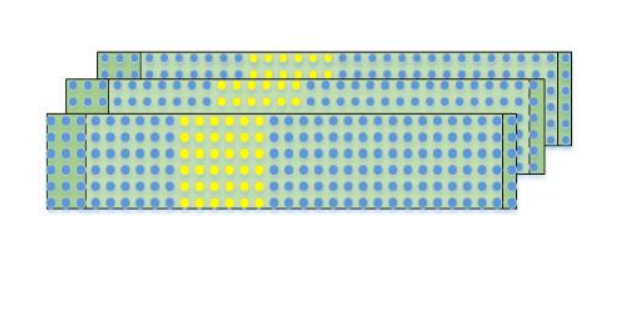}   
		\end{minipage}
	}
	
	\caption{Geometry of a 3D plate  under  transverse loading and its discretization.} 
	\label{Fig:4}
\end{figure}

A mesh with a size of $100 \times 30\times30$ is used to discrete this model. In order to investigate the convergence of the MTS scheme, the model was partitioned into two distinct subdomains. The first subdomain, represented by the yellow part in Fig. \ref{Fig:4} (b), corresponded to the area using $\Delta t_k$ with a thickness of $2\delta$. The second subdomain was computed using $\Delta t$. The time step size was refined from $1.0\times 10^{-5} \hspace{0.5em}\mathrm{s}$ to $0.125\times 10^{-5} \hspace{0.5em}\mathrm{s}$ to examine the accuracy of the scheme. To calculate the error, the displacement obtained from the UPD scheme when $\Delta t=1.0\times 10^{-8}$ was chosen as the exact solution. The error was then computed using the method outlined in Section 4.1.
\begin{table}[!ht]\begin{center}
		\caption{The error and order (CR) obtain by MTS scheme based on RK3 (MTS3) and RK4 (MTS4) with different $\Delta t$ and different refinement level $K$ on a 3D model}
		\vspace{0.15in}
		\setlength{\tabcolsep}{2mm}{
			\begin{tabular}{|c|c|l|c|l|c|c|c|} \hline
				\multicolumn{2}{|c|}{}&\multicolumn{2}{c|}{MTS3}&\multicolumn{2}{c|}{MTS4}\\ \hline
			$\Delta t$ (s)& $K$ & error & CR& error & CR\\ \hline				
				\multirow{4}*{\makecell[c]{\hspace{0.75em}$1.0$E$-5$}} 
				&1&$8.70$E$-08$&-&$7.72$E$-09$&- \\
				&2&$6.21$E$-08$&-&$5.48$E$-09$&-\\
				&4&$6.15$E$-08$&-&$5.46$E$-09$&-\\
				&8&$6.15$E$-08$&-&$5.46$E$-09$&-\\\hline
				\multirow{4}*{\makecell[c]{\hspace{0.75em}$0.5$E$-5$}} 
				&1&$1.14$E$-08$&3.12&$4.89$E$-10$&3.99\\
				&2&$8.13$E$-09$&3.13&$3.47$E$-10$&4.02\\
				&4&$8.06$E$-09$&3.13&$3.46$E$-10$&4.01\\
				&8&$8.06$E$-09$&3.13&$3.46$E$-10$&4.01\\\hline
				\multirow{4}*{\makecell[c]{\hspace{0.25em}$0.25$E$-5$}} 
				&1&$1.43$E$-10$&3.12&$3.07$E$-11$&4.00\\
				&2&$1.02$E$-10$&3.03&$2.17$E$-11$&4.00\\
				&4&$1.01$E$-10$&3.03&$2.17$E$-11$&4.00\\
				&8&$1.01$E$-10$&3.03&$2.17$E$-11$&4.00\\\hline
				\multirow{4}*{\makecell[c]{$0.125$E$-5$}} 
				&1&$1.78$E$-10$&3.00&$1.92$E$-12$&3.99\\
				&2&$1.27$E$-10$&3.01&$1.36$E$-12$&4.00\\
				&4&$1.26$E$-10$&3.00&$1.35$E$-13$&3.99\\
				&8&$1.26$E$-10$&3.00&$1.35$E$-13$&3.99\\\hline
		\end{tabular}}\label{tab:2}
\end{center}\end{table}

Table 2 demonstrates the accuracy and convergence of the three-dimensional PD model using the MTS scheme. One notable finding is that the order of the MTS scheme based on the RK method remains consistent with the order of the UPD scheme, regardless of the coarse and fine time step sizes. Furthermore, we observed that as the ratio of the coarse to fine mesh size increases, the displacement error decreases. For example, when $\Delta t = 0.5\times 10^{-5}\hspace{0.5em}\mathrm{s}$, the error decreases from $1.14\times 10^{-8}$ to $8.06\times 10^{-9}$ as the refinement level $K$ is increased from $1$ to $8$. This highlights the accuracy benefits of employing the MTS scheme.

\subsection{Plate with a pre-existing crack}

In this example, we consider a 2D plate with a length of $L = 0.05 \mathrm{~m}$ and a width of $W = 0.05 \mathrm{~m}$, which contains a pre-existing crack with a length of $2c = 0.01 \mathrm{~m}$, as shown in Fig. \ref{Fig:5} (a). The material properties, including density $\rho$, horizon size $\delta$, elastic modulus $E$, and Poisson's ratio $\mu$, are set to be consistent with Section 4.1. Velocity constraints are applied to $\mathcal{D}_c^1$ and $\mathcal{D}_c^2$ at the top and bottom of region $\mathcal{D}$, respectively. These constraints have a depth of $\delta$, as follows:
\begin{equation}
\begin{split}
\dot{u}_y(\mathbf{x}_i,t)&=\hspace{0.75em}20.0 \mathrm{~m} /\mathrm{s} \quad \mathbf{x}_i\in \mathcal{D}_c^1,\\
\dot{u}_y(\mathbf{x}_i,t)&=-20.0 \mathrm{~m} /\mathrm{s} \quad \mathbf{x}_i\in \mathcal{D}_c^2.
\end{split}   
\end{equation}

\begin{figure}[!ht]
	\centering  
	\subfigure[]  	
	{
		\begin{minipage}{6cm}
			\centering          
			\includegraphics[scale=0.28]{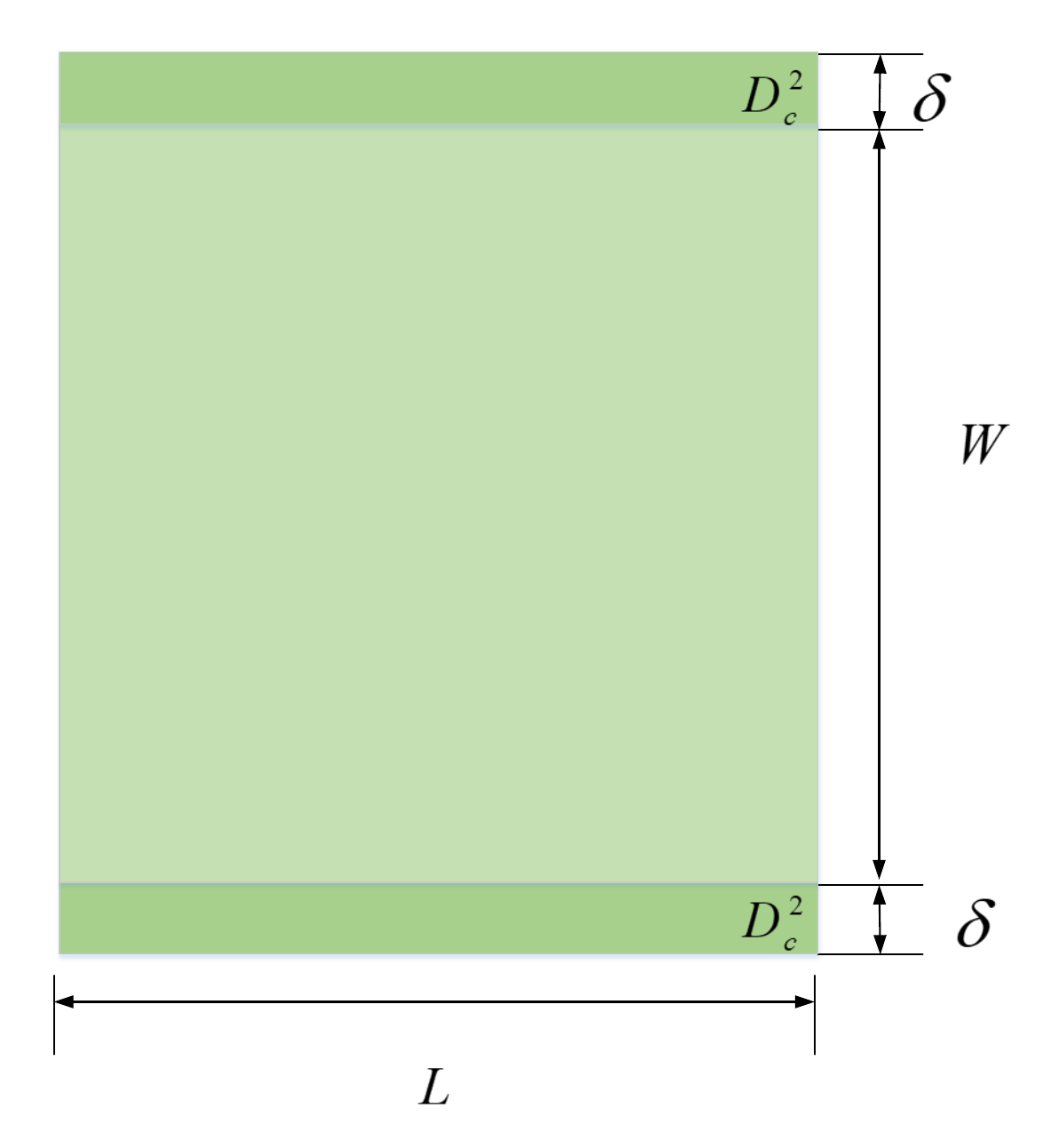}  
		\end{minipage}
	}
	\subfigure[]	
	{
		\begin{minipage}{6cm}
			\centering      
			\includegraphics[scale=0.28]{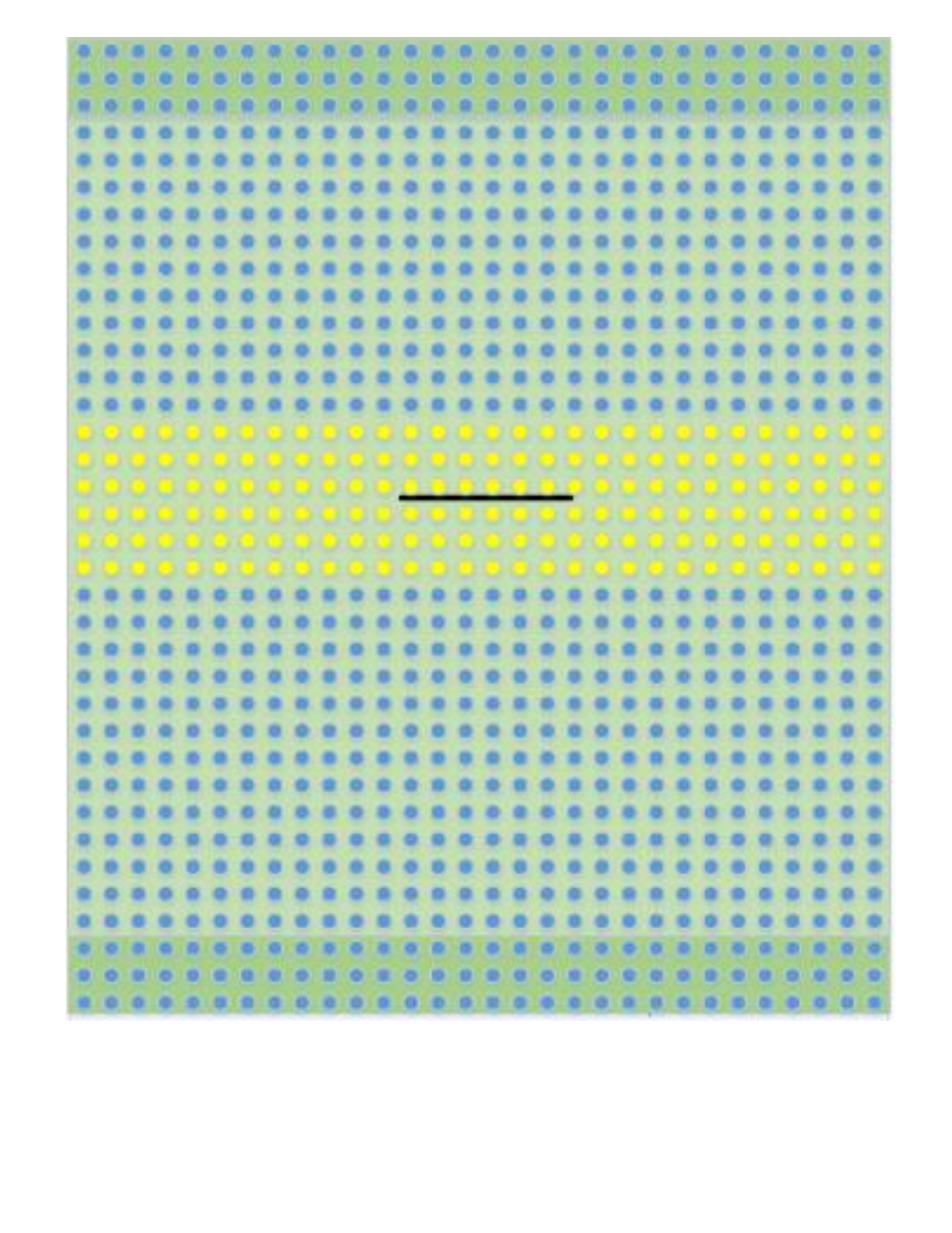}   
		\end{minipage}
	}
	
	\caption{Geometry of a 2D plate with a pre-existing crack .} 
	\label{Fig:5}
\end{figure}
For the discretization of the model, a mesh with a size of $500 \times 500$ is used. The history-dependent scalar function $\mu$ is used to identify fracture in the peridynamic simulation. The function is multiplied by equation (\ref{pd:e1}) to take into account the effects of previous deformation history on the material response and can be defined as follows:
\begin{equation}\label{mu2}
\mu(s,t)= \begin{cases} $1$, &\text{if the bond is not broken $s\textless s_0$}\\ $0$, & \text {if the bond is broken $s \ge s_0$}\end{cases}, 
\end{equation}
in which $s_0$ is the critical bond stretch and chosen as $0.01$, $s$ is the bond stretchs:
\begin{equation}\label{fra:s1}
s=\dfrac{|\bm{\eta}+\bm{\xi}|-|\bm{\xi}|}{|\bm{\xi}|}.
\end{equation}

In Fig. \ref{Fig:6}, the outcomes of the simulation are presented. The blue shades represent material points where the bonds remain intact despite the displacement changes, corresponding to $s=1$. On the other hand, the yellow shades indicate regions where the bonds have been broken, resulting in $s=0$. In certain areas where the material points have undergone significant stretching, the value of $s$ is between 0 and 1, indicating partial bond breakage. These results demonstrate that the MTS scheme approach accurately captures the displacement changes and fracture behavior, similar to the UPD scheme. The ability of the MTS scheme to handle complex deformation and fracture phenomena makes it a valuable tool in peridynamic simulations.

The MTS scheme possesses a noteworthy advantage, namely its capability to attain computational accuracy on par with the undecomposed PD solution while reducing computational time. This advantageous characteristic is demonstrated in Fig. 7. For a selected number of time steps, specifically 1500 with a time increment of $\Delta t=0.5\times10^{-5}\hspace{0.5em}\mathrm{s}$, the MTS scheme accomplishes the crack simulation in a mere 80 minutes of CPU time, compared to the 121 minutes required by the undecomposed PD scheme. The time saved in computation is directly linked to the extent of the fine time step region. Notably, the MTS scheme exhibits enhanced computational efficiency as the fine time step region decreases in size.

\begin{figure}[!ht]
	\centering	
	{\hspace{-1cm}
		\begin{minipage}{7cm}
			\centering          
			\includegraphics[scale=0.6]{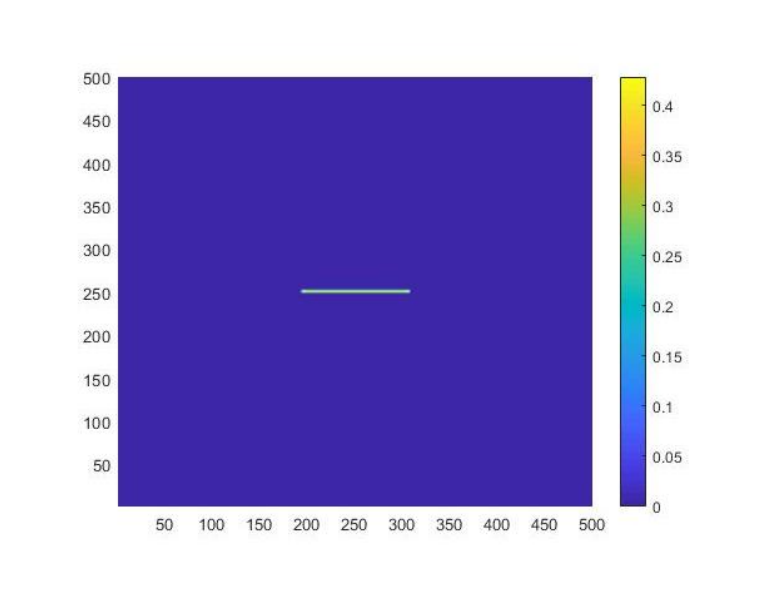}  
		\end{minipage}
	}	
	{
		\begin{minipage}{7cm}
			\centering      
			\includegraphics[scale=0.6]{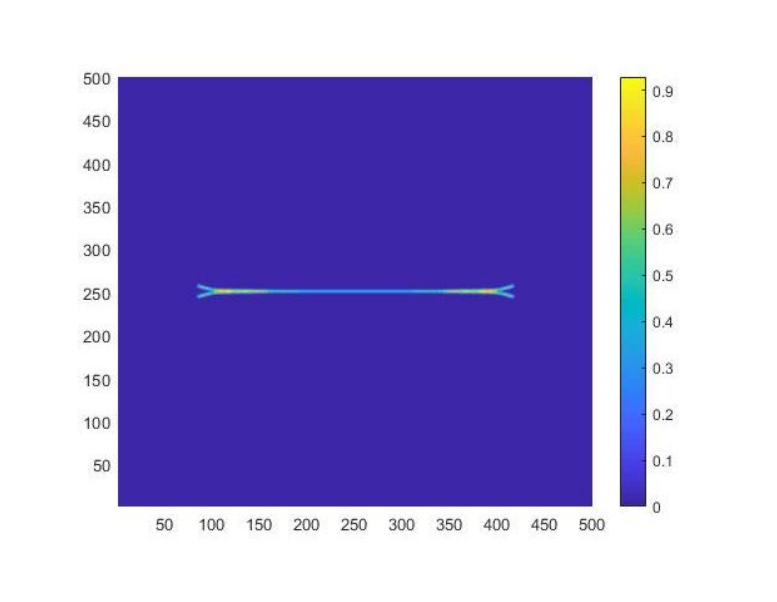}   
		\end{minipage}
	}
	{	\begin{minipage}{7cm}
			\centering      
			\includegraphics[scale=0.6]{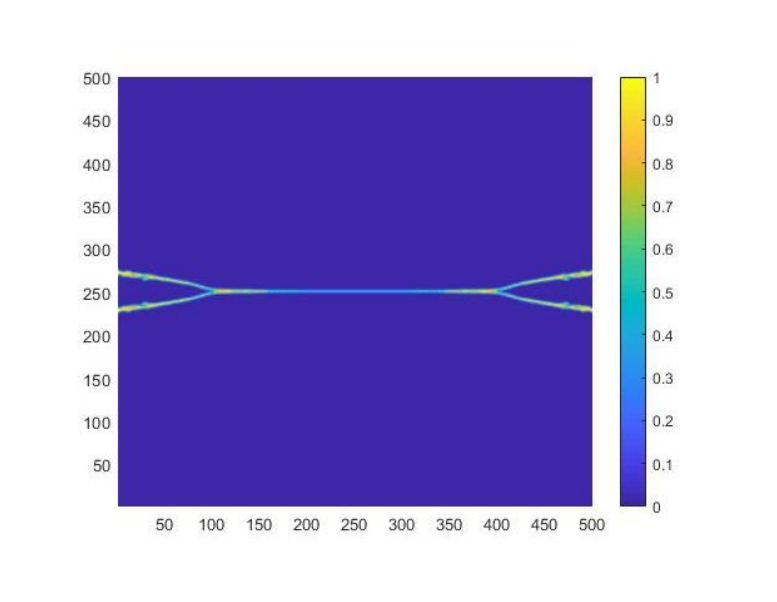}   
		\end{minipage}
	}	
	\caption{Crack propagation simulation by MTS scheme when the number of time step is chosen as $500$, $1000$ and $1500$ for $\Delta t=0.5\times10^{-5}\hspace{0.5em}\mathrm{s}$.} 
	\label{Fig:6}
\end{figure}

\begin{figure}[!ht]
	\centering   	
	{
		\begin{minipage}{7cm}
			\centering          
			\includegraphics[scale=0.48]{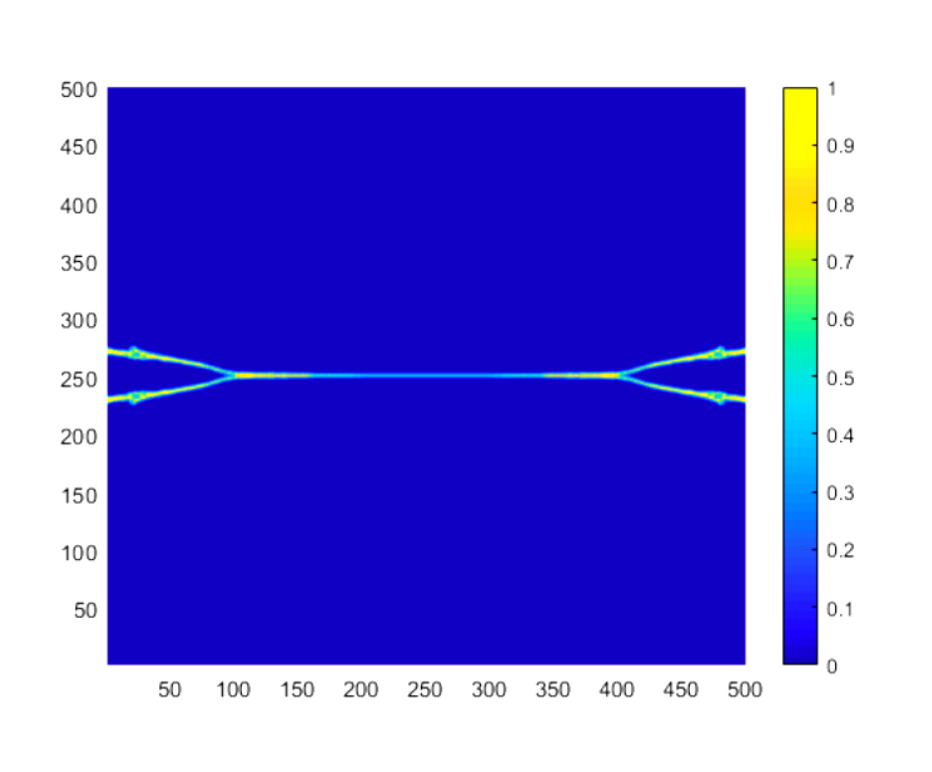}  
		\end{minipage}
	}	
	{
		\begin{minipage}{7cm}
			\centering      
			\includegraphics[scale=0.48]{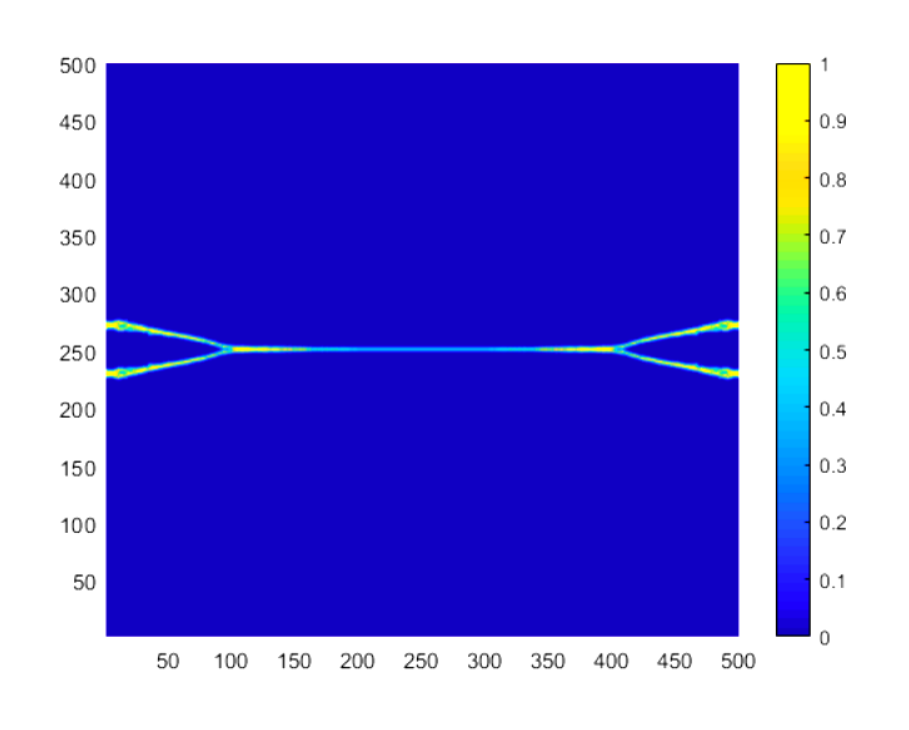}   
		\end{minipage}
	}	
	\caption{Crack propagation simulation by MTS scheme ($K$=2 and CPU time =80m) and undecomposed PD scheme ($K$=1 and CPU time =121m) when the time step is chosen as $1500$ for $\Delta t=0.5\times10^{-5}\hspace{0.5em}\mathrm{s}$}  
	\label{Fig:7}
\end{figure}

Table \ref{tab:3} provides a summary of the relative error, compared to the UPD solution, for the peridynamics approach with a time step size of $\Delta t = 0.125 \times 10^{-5}\hspace{0.5em}\mathrm{s}$. The reference displacement for error calculation is taken at $\Delta t= 1.0 \times 10^{-5}\hspace{0.5em}\mathrm{s}$. The table includes $e_c$, $e_f$, and $e$ as the errors for the coarse grid region, fine time step region, and the overall region, respectively.
The comparison between the UPD  and the MTS schemes for $K$ from $1$  to $8$. It should be noted that decreasing the time step size leads to improved accuracy in the calculations, as expected.

\begin{table}[!ht]\begin{center}
		\caption{The error of different subdomain obtain by the MTS scheme based on  RK4 with different $\Delta t$ , the number of time steps $nt$, and different  refinement level $K$ for a 2D crack model; $e_c$ is the errors for the coarse grid region , $e_f$ is the error for the fine grid region, and $e$ is the error for the overall region}
		\vspace{0.15in}
		\setlength{\tabcolsep}{2mm}{
			\begin{tabular}{|c|c|c|l|l|l|c|c|c|} \hline
				 	\multirow{2}*{$\Delta t$ (s)}&\multirow{2}*{$K$} &\multirow{2}*{}&\multicolumn{3}{c|}{$nt$}\\\cline{3-6}&&&$500$&$1000$&$1500$\\ \hline				
				\multirow{3}*{\makecell[c]{\hspace{0.75em}$1.0$E$-5$}}&\multirow{3}*{\makecell[c]{$8$}}&$e_f$&$5.43$E$-9$&$2.28$E$-7$&$2.88$E$-6$\\ 
				&&$e_c$&$8.68$E$-9$&$1.09$E$-7$&$4.91$E$-7$\\
				&&$e$&$1.03$E$-8$&$2.56$E$-7$&$2.92$E$-6$\\ \hline
				\multirow{3}*{\makecell[c]{\hspace{0.75em}$0.5$E$-5$}}&\multirow{3}*{\makecell[c]{$4$}}&$e_f$&$2.57$E$-9$&$1.38$E$-7$&$1.30$E$-6$\\ 
				&&$e_c$&$3.72$E$-9$&$5.05$E$-8$&$1.81$E$-7$\\
				&&$e$&$4.56$E$-8$&$1.48$E$-7$&$3.32$E$-6$\\ \hline
				\multirow{3}*{\makecell[c]{\hspace{0.25em}$0.25$E$-5$}}&\multirow{3}*{$2$}&$e_f$&$1.07$E$-9$&$8.13$E$-8$&$7.40$E$-7$\\ 
				&&$e_c$&$1.24$E$-9$&$2.20$E$-8$&$1.49$E$-7$\\
				&&$e$&$1.65$E$-9$&$8.46$E$-8$&$7.56$E$-7$\\ \hline
				\multirow{3}*{\makecell[c]{$0.125$E$-5$}}&\multirow{3}*{$1$}&$e_f$&$\hspace{2em}0$&$\hspace{2em}0$&$\hspace{2em}0$\\ 
				&&$e_c$&$\hspace{2em}0$&$\hspace{2em}0$&$\hspace{2em}0$\\
				&&$e$&$\hspace{2em}0$&$\hspace{2em}0$&$\hspace{2em}0$\\ \hline
		\end{tabular}}\label{tab:3}
\end{center}\end{table}

The MTS scheme demonstrates high accuracy in both coarse and fine time step sizes, with improved accuracy observed in the $\overline{\Omega}_C$ subdomain when comparing $K=1$ and $K=8$. In the fine time step region $\overline{\Omega}_F$, the MTS scheme outperforms the UPD scheme. Additionally, the MTS scheme can significantly speed up computational processes (up to four times in some examples) compared to the UPD scheme due to its flexible time step size divisions. Various examples have shown that the MTS scheme maintains high accuracy while efficiently solving problems of different sizes and time-step ratios. Hence, it is recommended for peridynamics simulations, particularly for large-scale problems.

\section{Conclusions}

The main contribution of this article is the introduction of a high-order MTS scheme for peridynamics. This scheme addresses the challenge of achieving computational accuracy by utilizing smaller spatial grids and time steps in discontinuous regions and larger spatial grids and time steps in other regions, thus overcoming the limitation of using a single time step for the entire simulation. The use of a higher-order Runge-Kutta method further enhances the accuracy and efficiency of the MTS scheme.
Theoretical analysis and numerical examples have been provided to validate the stability, accuracy, and computational advantages of the proposed method. Through a Taylor expansion analysis, it has been demonstrated that the MTS scheme maintains the order of the higher-order Runge-Kutta method. Several examples in both two-dimensional and three-dimensional settings have been employed to verify that the MTS schemes and the UPD scheme exhibit consistent convergence rates and achieve high accuracy at different time steps.

\section*{Acknowledgements}
H. Tian was supported by the Fundamental Research Funds for the Central Universities (202042008), the National Natural Science Foundation of China (11801533, 11971482). W.-S. Don would like to acknowledge the funding support of this research by the Shandong Provincial Natural Science Foundation (ZR2022MA012).
\bibliographystyle{elsarticle-num}
\bibliography{refer}
\end{document}